\documentclass
[
11pt]
{amsart}%

\usepackage{hyperref}
\usepackage{cite}
\usepackage{amssymb}
\usepackage{amsmath}
\usepackage{amsthm}
\usepackage{amsfonts}
\usepackage{graphicx}
\usepackage{bbm}
\usepackage{xcolor}
\usepackage[toc,page]{appendix}
\usepackage{subfigure}
\usepackage{mathtools}
\usepackage{bm} 
\usepackage[normalem]{ulem}
\usepackage{comment}
\usepackage{tikz-cd}

\newtheorem{theorem}{Theorem}[section]

\newtheorem{definition}[theorem]{Definition}
\newtheorem{notation}[theorem]{Notation}
\newtheorem{proposition}[theorem]{Proposition}
\newtheorem{corollary}[theorem]{Corollary}
\newtheorem{lemma}[theorem]{Lemma}
\newtheorem{remark}[theorem]{Remark}

\newtheorem*{acknowledgement}{Acknowledgement}

\numberwithin{equation}{section}

\newcommand{\R}{\mathbb{R}}  
\newcommand{\E}{\mathbb{E}} 
\newcommand{\Prob}{\mathbb{P}}
\newcommand{\Q}{\mathbb{Q}}
\newcommand{\Hei}{\mathbb{H}}

\newcommand{\WR}{W_0 \left(\mathbb{R}^2\right)}
\newcommand{\WH}{W_0 \left( \mathbb{H} \right)}

\DeclareMathOperator{\Span}{span}

\begin{document}


\title[Support Hypoelliptic Heisenberg group]{On the Support of a hypoelliptic diffusion on the Heisenberg group}

\author[Carfagnini]{Marco Carfagnini{$^{\dag }$}}
\thanks{\footnotemark {$\dag$} Research was supported in part by NSF Grants DMS-1712427 and DMS-1954264. }
\thanks{ \textit{Data availability statement:} no datasets were generated or analyzed during the current study.}
\address{ Department of Mathematics\\
University of California San Diego\\
La Jolla, CA 92093-0112,  U.S.A.}
\email{mcarfagnini@ucsd.edu}

\keywords{Diffusion processes, Wiener measure, Heisenberg group, hypoelliptic operator}

\subjclass{Primary 58J65, 60H10; Secondary 60J60, 60H05 }


\begin{abstract} 
We provide an elementary proof of the support of the law of a hypoelliptic Brownian motion on the Heisenberg group $\Hei$. We consider a control norm associated to left-invariant vector fields on $\Hei$, and describe the support in terms of the space of finite energy horizontal curves. 
\end{abstract}

\maketitle

\tableofcontents

\renewcommand{\contentsname}{Table of Contents}

\maketitle

\tableofcontents

\section{Introduction}\label{Intro}

The purpose of this paper is to describe the support of the law of a hypoellptic diffusion $g_{t}$ on the Heisenberg group $\Hei$. The novelty of this paper is the norm on the path space $\WH$ of $\Hei$-valued continuous curves starting at the identity that is used for the support. The group $\Hei$ is the simplest example of a Carnot group and it comes with a natural left-invariant distance, the Carnot-Carath\'{e}odory distance $d_{cc}$. This is the control distance associated to the left-invariant vector fields on $\Hei$, see Definition \ref{Dfn.2.2}. The corresponding control norm $d_{c}$  is defined as $d_{c}(x):= d_{cc} (x,e)$, where $e\in \Hei$ is the identity.  Our main result is to prove a support theorem for the hypoelliptic Brownian motion $g_{t}$ with respect to the norm $\max_{0\leqslant t \leqslant 1} d_{c} (\gamma (t))$ for $\gamma \in \WH$. As pointed out in Remark \ref{r.norms} it is sufficient to describe the support  with respect to some equivalent norms. 

The support of a diffusion was first studied by Stroock and Varadhan \cite{StroockVaradhan1972}, which we now describe briefly. Let $X_{t}$ be an $\R^{d}$-valued diffusion which is solution to the stochastic differential equation 
\begin{equation}\label{e.SDE.strat}
dX_t = \sigma \left( t, X_t \right)\circ dW_t + b\left( t,X_t \right)dt, \quad X_0=0,
\end{equation}
where $\sigma (t,x)$ is a $d\times \ell$ matrix whose entries are functions of $(t,x)\in [0,1] \times \R^d$, and $b(t,x)$ is a vector in $\R^d$, and $ W_{t}$ is an $\ell$-dimensional Brownian motion, and  $\circ d$ denotes the stochastic differential in Stratonovich's sense. We can view  the process $\{X_{t} \}_{0\leqslant t \leqslant 1}$ as a $W_{0}(\R^{d})$-valued random variable, where $W_{0}(\R^{d})$ is the space of $\R^d$-valued continuous paths starting at zero. Let $\mu$ be the law of  $\{X_{t} \}_{0\leqslant t \leqslant 1}$ and $\mathcal{S}_{\mu}$ its support. If $H$ denotes the subset of $W_{0}(\R^{d})$ consisting of absolutely continuous paths, then to any $\phi\in H$ one can associate a deterministic path $x_\phi $ as being the solution to the ordinary differential equation 
\begin{align}\label{e.controlled.system}
& x^\prime_{\phi} (t)  = \sigma \left( t, x_\phi (t) \right) \phi^\prime (t) dt + b\left( t, x_\phi (t) \right) dt, 
\\
&x_\phi(0) = 0. \notag
\end{align}
We follow \cite{Kunita1978} and refer to solutions to \eqref{e.controlled.system} as controlled systems. Then 
\begin{equation}\label{e.support}
\mathcal{S}_\mu = \overline{ \left\{ x_\phi, \; \phi \in H \right\} }^{\infty},
\end{equation}
where the closure is taken in the uniform topology in $W_{0}(\R^d)$. 

Note that the hypoelliptic Brownian motion $g_{t}$ can be viewed as an $\R^{3}$-valued stochastic process. This is not a Gaussian process and it satisfies the stochastic differential equation 
\begin{align*}
\begin{pmatrix}
dg_{1}(t)
\\
dg_{2} (t)
\\
dg_{3} (t) 
\end{pmatrix} 
=
\begin{pmatrix}
1 & 0
\\
0 & 1
\\
- \frac{1}{2}  B_{2}(t) & \frac{1}{2} B_{1}(t) 
\end{pmatrix} 
\begin{pmatrix}
dB_{1}(t) 
\\
dB_{2} (t) 
\end{pmatrix} 
, \quad g_{0} = (0,0,0).
\end{align*}
Stroock-Varadhan original support theorem \eqref{e.support} was proven under the assumption that $\sigma$ is $C^{2}$ in space and $C^{1}$ in time, bounded together with its partial derivatives of order one and two, and $b$ is globally Lipschitz and bounded.  In a series of papers by Gy\"{o}ngy   \cite{Gyongy1988, Gyongy1988b, Gyongy1994}, and by Gy\"{o}ngy-Pr\"{o}hle \cite{GyongyProhle1990} \eqref{e.support} is proven for processes driven by continuous semi-martingales under milder assumptions on $\sigma$ and $b$. In particular, \eqref{e.support} for the law of $\{g_{t}\}_{0\leqslant t \leqslant 1}$ with respect to the uniform norm $\max_{0\leqslant t \leqslant 1} \vert \gamma (t) \vert_{\R^{3}}$ on $W_{0} ( \R^{3})$ follows from  \cite[Theorem 3.1]{GyongyProhle1990}.  Moreover, \eqref{e.support} for the hypoelliptic Brownian motion can also be proven by rough path theory and continuity of the Lyons-It\^o map, that is, the solution map of a rough differential equation \cite[Section 10.3, Section 13.7]{FrizVictoirBook}. In the current paper we prove \eqref{e.support} for the hypoelliptic Brownian motion on the Heisenberg group. Differently from \cite{GyongyProhle1990}, we replace the Euclidean norm in $\R^{3}$ by the control norm $d_{c}$, which is a more natural norm and it is consistent with the geometry of the Heisenberg group. Our proof does not rely on rough path theory and it is based on a time change argument. 

We mention that \eqref{e.support} for diffusion processes on Hilbert spaces is proven  in \cite{Aida1990, Gyongy1989}. We also mention that in \cite{LedouxQianZhang2002} a rough paths approach  is used, and a support theorem in the $p$-variational topology is proven. 

One can ask under what condition the closure in \eqref{e.support} coincides  with the whole path space $W_{0}(\R^{d})$. This question has been addressed in \cite{Kunita1978}, where the author gives nearly necessary and sufficient  conditions for 
\begin{equation}\label{e.kunita}
W_{0}(\R^{d}) = \overline{ \left\{ x_\phi, \; \phi \in H \right\} }^\infty
\end{equation}
to hold. We prove \eqref{e.kunita} for the hypoelliptic Brownian motion. Our proof is explicit and it relies on the group structure of $\Hei$. 

The main results of this paper are contained in Theorem \ref{Thm.Support}, where we prove \eqref{e.support} and \eqref{e.kunita} for the hypoelliptic Brownian motion on the Heisenberg group. More precisely, if $H \left( \Hei \right)$  denotes the set of finite energy horizontal curves, then we prove that 
\vspace{-0.2cm}
\begin{equation}\label{eqn.main}
\WH = \overline{H (\Hei )}^{d_{c}} = \mathcal{S}_{\mu},
\end{equation}
where $\mu$ is the law of $\{g_{t} \}_{0 \leqslant t \leqslant 1}$ and the closure is taken with respect to the norm $\max_{0\leqslant t \leqslant 1} d_{c} (\gamma (t) )$ for $\gamma \in \WH$.
\vspace{-0.1cm}

First, we show that $\mathcal{S}_{\mu} \subset \overline{H (\Hei )}^{d_{c}}$ by constructing a family of stochastic processes $\{ g_{\delta}\}_{\delta >0}$ that approximates $g$ in the sense that the law $\mu_{\delta}$ of $g_{\delta}$ converges weakly to the law $\mu$ of $g$. We further study relations between the measures $\mu_{\delta}$  and $\mu$, and prove that they are singular. In particular we show that the space $H(\Hei)$ of finite energy horizontal curves has $\mu$-measure zero. The space $H(\Hei)$ can then be viewed as a Cameron-Martin space in a non-Gaussian setting by \eqref{eqn.main} and since  $\mu (H (\Hei ) ) =0$.  We then prove that  $\overline{H (\Hei )}^{d_{c}} \subset \mathcal{S}_{\mu}$ using Theorem \ref{thm.estimate} and the explicit form of the process $g_{t}$.  Namely, $g_{t} = \left( B_{t}, A_{t} \right) $, where $B_{t}$ is a two-dimensional standard Brownian motion and $A_{t}$ is the corresponding L\'{e}vy's stochastic area. Our proof relies on the classical identity $A_{t} = b_{\tau (t)}$, where $b_{t}$ is a one-dimensional standard Brownian motion independent of $B_{t}$, and $\tau (t)$ is a stopping time. 

The paper is organized as follows. In Section \ref{sec2} we describe the Heisenberg group $\Hei$, and the corresponding control norm and hypoelliptic Brownian motion, and state the main result of the paper Theorem \ref{Thm.Support}. Section \ref{sec3} then contains the proof of Theorem \ref{Thm.Support}.

\section{The setting and the main result}\label{sec2}
\subsection{Heisenberg group as Lie group}\label{subsec3.1}
The Heisenberg group $\Hei$ as a set is  $\R^3\cong \mathbb{R}^{2} \times \mathbb{R}$ with the group multiplication  given by
\begin{align*}
& \left( \mathbf{v}_{1}, z_{1} \right) \cdot \left( \mathbf{v}_{2}, z_{2} \right) := \left( x_{1}+x_{2}, y_{1}+y_{2}, z_{1}+z_{2} + \frac{1}{2}\omega\left( \mathbf{v}_{1}, \mathbf{v}_{2} \right)\right),
\\
& \text{ where } \mathbf{v}_{1}=\left( x_{1}, y_{1} \right), \mathbf{v}_{2}=\left( x_{2}, y_{2} \right) \in \mathbb{R}^{2}, \; \text{ and }
\\
& \omega: \mathbb{R}^{2} \times \mathbb{R}^{2} \longrightarrow \mathbb{R}, \quad \omega\left( \mathbf{v}_{1}, \mathbf{v}_{2} \right):= x_{1}y_{2}-x_{2} y_{1}.
\end{align*}
The identity in $\Hei$ is $e=(0, 0, 0)$ and the inverse is given by $\left( \mathbf{v}, z \right)^{-1}= (-\mathbf{v},-z)$. The Lie algebra of $\Hei$ can be identified with the space  $\R^3\cong \mathbb{R}^{2} \times \mathbb{R}$  with the Lie bracket defined by
\[
\left[ \left( \mathbf{a}_{1}, c_{1} \right), \left( \mathbf{a}_{2}, c_{2} \right)  \right] = \left(0,\omega\left( \mathbf{a}_{1}, \mathbf{a}_{2} \right)  \right).
\]
The set $\R^3\cong \mathbb{R}^{2} \times \mathbb{R}$ with this Lie algebra structure will be denoted by $\mathfrak{h} $.

Let us now recall some basic notation for Lie groups. Suppose $G$ is a Lie group, then the left  and right multiplication by an element $k\in G$ are denoted by
\begin{align*}
L_{k}: G \longrightarrow G, &  & g \longmapsto k^{-1}g,
\\
R_{k}: G \longrightarrow G, &  & g \longmapsto gk.
\end{align*}

Recall that  the tangent space $T_{e}G$ can be identified with the Lie algebra $\mathfrak{g}$ of left-invariant vector fields on $G$, that is, vector fields $X$ on $G$  such that $dL_{k} \circ X=X \circ L_{k}$, where $dL_{k}$ is the differential of $L_k$. More precisely, if $A$ is a vector in $T_{e}G$, then we denote by $\tilde{A}\in \mathfrak{g}$ the (unique) left-invariant vector field such that $\tilde{A} (e) = A$.  A left-invariant vector field is determined by its value at the identity, namely, $\tilde{A}\left( k \right)=dL_{k} \circ\tilde{A}\left( e \right)$.

For the Heisenberg group the differential of left and right multiplication can be described explicitly as follows.

\begin{proposition}\label{p.Differentials}
Let $k= (k_1, k_2, k_3) = (\mathbf{k}, k_3 )$ and $g= (g_1, g_2, g_3) = (\mathbf{g}, g_3 )$ be two elements in $\Hei$. Then, for every $v= \left( v_1, v_2, v_3 \right) = (\mathbf{v}, v_3 )$ in  $T_g\Hei$, the differentials of the left and right multiplication are given by
\begin{align}\label{LeftRightMultDiff}
& dL_{k}: T_g\Hei \longrightarrow T_{k^{-1}g}\Hei,  \notag
\\
& dR_{k}: T_g\Hei \longrightarrow T_{gk}\Hei, \notag
\\
& dL_{k} (v) =  \left( v_1, v_2, v_3 + \frac{1}{2} \omega( \mathbf{v}, \mathbf{k}) \right), \notag \\
& dR_{k} (v) =  \left( v_1, v_2, v_3 + \frac{1}{2} \omega( \mathbf{v}, \mathbf{k}) \right).
\end{align}
\end{proposition}

\subsection{Heisenberg group as a sub-Riemannian manifold}
The Heisenberg group $\Hei$ is the simplest non-trivial example of a sub-Riemannian manifold.
We define $X$, $Y$ and $Z$ as the unique left-invariant vector fields satisfying $X_e = \partial_x$, $Y_e = \partial_y$ and $Z_e = \partial_z$, that is,
\begin{align*}
& X = \partial_x - \frac{1}{2}y\partial_z,
\\
& Y = \partial_y + \frac{1}{2}x\partial_z,
\\
&  Z = \partial_z.
\end{align*}
Note that the only non-zero Lie bracket for these left-invariant vector fields is $[X, Y]=Z$, so the vector fields $\left\{ X, Y \right\}$ satisfy H\"{o}rmander's condition. We define the \emph{horizontal distribution} as $\mathcal{H}:= \Span \left\{  X, Y \right\}$ fiberwise, thus making $\mathcal{H}$ a sub-bundle in the tangent bundle $T\Hei$. To finish the description of the Heisenberg group as a sub-Riemannian manifold we need to equip the horizontal distribution  $\mathcal{H}$ with an inner product. For any $p \in \Hei$ we define the inner product $\langle \cdot , \cdot \rangle_{\mathcal{H}_{p}}$ on $\mathcal{H}_{p}$ so that $\left\{ X \left( p \right), Y \left( p \right) \right\}$ is an orthonormal (horizontal) frame at any $p \in \Hei$. Vectors in $\mathcal{H}_{p}$ will be called \emph{horizontal}, and the corresponding norm will be denoted by $\Vert \cdot \Vert_{\mathcal{H}_{p}}$.

In addition, H\"{o}rmander's condition ensures that a natural sub-Laplacian on the Heisenberg group

\begin{equation}\label{e.2.1}
\Delta_{\mathcal{H}} =  X ^2 + Y ^2
\end{equation}
is a hypoelliptic operator by \cite{Hormander1967a}. We recall now another important object in sub-Riemannian geometry, namely, horizontal curves.

\begin{notation}\label{bold}
A curve $\gamma(t) = \left( x\left( t \right), y\left( t \right), z\left( t \right) \right)$ in $\Hei$ will be denoted by $\left( \mathbf{x}\left( t \right), z\left( t \right) \right)$, and its corresponding tangent vector $\gamma^\prime(t)$ in $T\Hei_{\gamma(t)} $ will be denoted by
\[
\gamma^\prime (t) = \left( x^\prime\left( t \right), y^\prime\left( t \right), z^\prime\left( t \right) \right)=\left( \bm{x}^\prime\left( t \right), z^\prime\left( t \right) \right).
\]
\end{notation}

\begin{definition}\label{Dfn.2.1}
An absolutely continuous path $t \longmapsto \gamma (t) \in \Hei$, for a.e. $t \in [0,1]$ is said to be horizontal if $\gamma^{\prime}(t)\in\mathcal{H}_{\gamma(t)}$ for a.e. $t$, that is, the tangent vector to $\gamma\left(t\right)$ at every point $\gamma \left(t\right)$ is horizontal. Equivalently we can say that $\gamma$ is horizontal if  $c_{\gamma}\left( t \right):=dL_{\gamma\left( t \right)}\left( \gamma^{\prime}(t)\right) \in \mathcal{H}_{e}$ for a.e. $t$.
\end{definition}
Note that for $\gamma(t) = \left( \mathbf{x}\left( t \right), z\left( t \right) \right)$ we have
\begin{align}\label{e.MaurerCartan}
& c_{\gamma}\left( t \right)=dL_{\gamma\left( t \right)}\left( \gamma^{\prime}(t)\right)=\left( \mathbf{x}^\prime\left( t\right), z^{\prime}\left( t \right) -\frac{1}{2}\omega( \mathbf{x}\left( t \right), \mathbf{x}^{\prime}\left( t \right) )\right), 
\end{align}
where we used Proposition \ref{p.Differentials}. Equation \eqref{e.MaurerCartan} can be used  to characterize horizontal curves in terms of the components as follows. The curve $\gamma$ is horizontal if and only if, for a.e. $0\leqslant t \leqslant 1$
\begin{equation}\label{e.horizontal}
z^{\prime}(t) -\frac{1}{2}\omega( \mathbf{x}\left( t \right), \mathbf{x}^{\prime}\left( t \right) ))=0.
\end{equation}

\begin{definition}\label{dfn.finite.energy.path}
We say that a horizontal curve $t \longmapsto \gamma (t) \in \Hei, \, 0\leqslant  t\leqslant 1$ has finite energy if
\begin{equation}\label{e.finite.energy}
\Vert \gamma \Vert_{H\left( \Hei \right)}^{2}
:= \int_0^1 \vert c_\gamma \left( s \right) \vert^2_{\mathcal{H}_e}  ds=\int_0^1 \vert dL_{\gamma(s)} \left( \gamma^\prime(s) \right)\vert^2_{\mathcal{H}_{e}} ds < \infty.
\end{equation}
\end{definition}
We denote  by $H\left( \Hei \right)$ the space of finite energy horizontal curves starting at the identity.
The inner product corresponding to the norm $\Vert \cdot \Vert_{H\left( \Hei \right)}$ is denoted by $\langle \cdot, \cdot \rangle_{H\left( \Hei \right)}$.  Note that the Heisenberg group as a sub-Riemannian manifold comes with a natural left-invariant distance.

\begin{definition}\label{Dfn.2.2}
For any $x, y \in \Hei$ the Carnot-Carath\'{e}odory distance is defined as
\begin{align*}
d_{cc} (x, y):= &\inf \left\{  \int_0^1  \vert c_\gamma \left( s \right) \vert_{\mathcal{H}_e} ds , \right.
\\
& \left. \gamma : [0,1] \longrightarrow \Hei, \gamma(0)=x, \gamma(1)=y,  \gamma  \text{ is horizontal}  \right\}.
\end{align*}
\end{definition}
Another consequence of  H\"{o}rmander's condition for left-invariant vector fields $X$, $Y$ and $Z$ is that we can apply the Chow–Rashevskii theorem. As a result, given two points in $\Hei$ there exists a horizontal curve connecting them, and therefore the Carnot-Carath\'{e}odory distance is finite on $\Hei$. The Carnot-Carath\'{e}odory distance in Definition \ref{Dfn.2.2} is the control distance associated to the vector fields $X,Y$, and $Z$ on $\Hei$ \cite[Definition 5.2.2]{BonfiglioliLanconelliUguzzoniBook}. The control norm $d_{c} : \Hei \rightarrow \R$ is then defined as $d_{c} (x) = d_{cc} (x,e)$. Note that 
\begin{equation}\label{eqn.left.invariance.distance}
d_{c} ( y^{-1}x) = d_{cc} (x,y),
\end{equation}
by left-invariance of $X,Y$, and $Z$, and the definition of $d_{cc}$. The control norm is an example of a homogeneous norm. 

\begin{definition}\label{def.hom.norm}
Let $\rho : \Hei \rightarrow [0,\infty)$ be a continuous function with respect to the Euclidean topology. Then $\rho$ is a homogeneous norm if it satisfies the following properties
\begin{align*}
& \rho (\delta_{\lambda} (x) ) = \lambda \rho (x), \text{ for every } \lambda >0, \text{ and } x \in \Hei,
\\
& \rho (x) = 0 \text{ if and only if } x=e,
\end{align*}
where $\delta_{\lambda} (x):= \left( \lambda x_{1}, \lambda x_{2} , \lambda^{2} x_{3} \right)$.
\end{definition}

If $\rho_{1}$ and $\rho_{2}$ are two homogeneous norms, then there exists a constant $c>0$ such that 
\begin{equation}\label{e.DistEquivalence}
c^{-1} \rho_{1} (x) \leqslant \rho_{2} (x) \leqslant c \rho_{1}(x) ,
\end{equation}
for every $x\in \Hei$, \cite[Proposition 5.1.4]{BonfiglioliLanconelliUguzzoniBook}. We consider the following homogeneous norm 
\begin{equation}\label{hom.norm}
\vert x \vert:= \left(  (x_{1}^{2} +x_{2}^{2})^{2} + x_{3}^{2} \right)^{\frac{1}{4}},
\end{equation}
for every $x= (x_{1}, x_{2}, x_{3} )\in \Hei$. By \eqref{eqn.left.invariance.distance} and \eqref{e.DistEquivalence} it follows that 
\begin{align*}
c^{-1}  d_{cc} (x,y) \leqslant \vert y^{-1} x \vert  \leqslant c d_{cc} (x,y),
\end{align*} 
for any $x,y \in \Hei$.

Finally, we need to describe  a hypoelliptic Brownian motion with values in $\Hei$. 

\begin{definition}\label{d.HeisenbergBM} 
An $\Hei$-valued Markov process $g_{t}$  is called a hypoelliptic Brownian motion if its generator is the sub-Laplacian $\frac{1}{2}\Delta_{\mathcal{H}}$ defined by Equation \eqref{e.2.1}.
\end{definition}
One can write a stochastic differential equation for $g_{t}$. This form is the standard stochastic differential equation for a Lie group-valued Brownian motion, namely,
\begin{align*}
&dL_{g_{t}}\left( dg_{t} \right)=(dB_{1}(t), dB_{2} (t), 0),
\\
& g_{0}=e,
\end{align*}
where $B_{t} = ( B_{1} (t), B_{2}(t))$ is a standard two-dimensional Brownian motion. An explicit solution is given by 
\begin{equation}\label{e.HypoBM}
g_{t}:=\left( B_{t}, A_{t} \right),
\end{equation}
where $A_t := \frac{1}{2} \int_0^t \omega\left( B_{s}, dB_{s}\right)$ is the Levy's stochastic area. Note that we used the It\^{o} integral rather than the Stratonovich integral. However, these two integrals are equal since the symplectic form $\omega$ is skew-symmetric, and therefore  L\'{e}vy's stochastic area functional is the same for both integrals.

\begin{notation}\rm\label{ProbabilisticSetting}
Throughout the paper we fix a filtered probability space $\left( \Omega, \mathcal{F}, \mathcal{F}_{t}, \mathbb{P}\right)$. We denote the expectation under $\Prob$  by $\E$.
\end{notation}

\subsection{The Wiener meaure} 
We recall here the definition of Wiener measure, and collect some notations that will be used throughout the paper. 
\begin{notation}[Topology on $\Hei$]
Let $\vert \cdot \vert$ be the homogeneous norm in \eqref{hom.norm}. We consider the topology on $\Hei$ whose open balls centered at the identity   are $\left\{ x\in\Hei, \,\: \vert x\vert < r  \right\} $.
\end{notation}

Note that by \eqref{e.DistEquivalence} all homogeneous norms induce the same topology.

\begin{notation}[Standard Wiener space] We denote by $W_{0}\left( \mathbb{R}^{n} \right)$ the space of $\mathbb{R}^{n}$-valued continuous functions on $[0,1]$ starting at $0$. This space comes with the norm

\[
\Vert h \Vert_{W_0 \left(\R^n\right)}:=\max_{0\leqslant t \leqslant 1} \vert h (t) \vert_{\R^n}, \quad h\in W_0 \left( \R^n \right),
\]
and the associated distance $d_{W_0 \left(\R^n\right)} (h,k) =\max_{0 \leqslant t\leqslant 1} \vert h (t) - k(t) \vert_{\R^n}   $, where $\vert \cdot \vert_{\R^n} $ is the Euclidean norm. 

\end{notation}

\begin{definition}[Wiener space over $\Hei$]\label{d.wiener.space}
The Wiener space over $\Hei$, denoted by $\WH$, is the space of $\mathbb{H}$-valued continuous functions on $[0,1]$  starting at identity in $\mathbb{H}$.
\end{definition}
Once a homogeneous norm $\rho$ on $\Hei$ is fixed, one can introduce a topology on $\WH$ in the following way. We endow $W_0 \left( \Hei \right)$ with the following norm
\[
\Vert \eta \Vert_{\rho}:=\max_{0\leqslant t \leqslant 1} \rho ( \eta (t) ), \quad \eta \in W_0 \left( \Hei \right),
\]
and the associated distance is $\max_{ 0 \leqslant t \leqslant 1 } \rho ( \eta(t)^{-1} \gamma (t) )$ for any $\eta, \gamma \in W_0 \left( \Hei \right)$.

\begin{definition}\label{d.hypo.wiener.measure} 
Let $W_0 \left( \Hei \right)$ be the Wiener space over $\Hei$, and $\{g_{t}\}_{0\leqslant t \leqslant 1}$ be the hypoelliptic Brownian motion defined by  \eqref{e.HypoBM}. We call its law the horizontal Wiener measure and we denote it by $\mu$.
\end{definition}
The process $g_t$ can be viewed as a $\WH$-valued random variable, that is,
\begin{align*}
&g\, : \Omega \longrightarrow \WH, \qquad \omega \mapsto \left\{ t \mapsto g_t(\omega) \right\}.
\end{align*}
The measure $\mu$ is then given by  $\mu (E) = \Prob \left( g^{-1} (E) \right) = \Prob \left( g \in E \right) $ for any Borel set $E$ in $\WH$.
We denote the support of $\mu$ by $\mathcal{S}_{\mu} $, that is, $\mathcal{S}_{\mu}$ is the smallest closed subset of $\WH$ having $\mu$-measure one.

\begin{remark}\label{r.not.banach.space}
Note that even though the hypoelliptic Brownian motion $g_t$ is an $\R^3$-valued stochastic process, it is not a Gaussian process, and its law $\mu$ is not a Gaussian measure on $\WH$. Moreover, contrary to the Euclidean case, the space $\WH$ is not a Banach space. It is easy to see that the space $\WH$ is closed under the norm $ \max_{0 \leqslant t \leqslant 1} \rho (\gamma (t) )$ for $\gamma \in \WH$, where $\rho$ is a homogeneous norm on $\Hei$, but $\WH$ is not a linear space.
\end{remark} 

Let us denote by $\pi$ the projection 
\begin{equation}\label{eqn.projection}
\pi : \WH \longrightarrow W_{0} ( \R^{2}), \; \pi (\gamma) = ( \gamma_{1} , \gamma_{2}),
\end{equation}
for any $\gamma = (\gamma_{1} , \gamma_{2} , \gamma_{3} ) \in \WH$.

\begin{remark}\label{r.explanation.name}
Let $\phi = \left( \phi_1, \phi_2, \phi_3 \right) \in H\left( \Hei\right)$ be a finite energy horizontal curve as in Definition \ref{dfn.finite.energy.path}. Then $\pi(\phi)$ is in the Cameron-Martin space on $\R^{2}$, that is, $\pi(\phi)$ is an absolutely continuous $\R^{2}$-valued curve starting at zero such that 
\[
\int_0^1 \vert \pi( \phi)^{\prime} (s) \vert_{\R^2 }^2ds < \infty.
\]
\end{remark}

\subsection{Main result} 
Now we have all the ingredients needed to state the main result of this paper, that is, we describe the support of the Wiener measure for the hypoelliptic Brownian motion $g_t$ in terms of horizontal paths.

\begin{theorem}\label{Thm.Support}
Let $W_0 \left( \Hei \right)$ be the Wiener space over $\Hei$, and $\mu$ be the horizontal Wiener measure on $\WH$, and $H\left( \Hei \right)$ be  the space of horizontal curves with finite energy. Then 
\[
\mathcal{S}_{\mu} = \overline{H\left( \Hei\right) }^{d_{c} } = \WH,
\]
where the closure is taken with respect to the norm $\max_{0 \leqslant t \leqslant 1} d_{c} (\gamma (t) )$, for $ \gamma \in \WH$, and $d_{c}$ is the control norm induced by the Carnot-Carath\'eodory distance.
\end{theorem}

\begin{remark}\label{r.norms}
It is enough to prove Theorem \ref{Thm.Support} for the homogeneous norm $\vert \cdot \vert$ given by \eqref{hom.norm}. Indeed, if $\rho_{1}$ and $\rho_{2}$ are two homogeneous norms on $\Hei$ and $\Vert \cdot \Vert_{\rho_{1}}$, $\Vert \cdot \Vert_{\rho_{2}}$ denote the corresponding norms on $\WH$, that is, 
\[ 
\Vert \gamma \Vert_{\rho_{i}} := \max_{0 \leqslant t \leqslant 1} \rho_{i} (\gamma (t) ), \quad \gamma \in \WH.
\]
Then 
\[
\overline{H\left( \Hei\right) }^{\rho_{1} }= \overline{H\left( \Hei\right) }^{\rho_{2} },
\]
since $\Vert \cdot \Vert_{\rho_{1}}$, $\Vert \cdot \Vert_{\rho_{2}}$ are equivalent by \eqref{e.DistEquivalence}, and hence Theorem \ref{Thm.Support} holds for any homogeneous norm as soon as it holds for one norm. 
\end{remark}

Let us denote by $\overline{H\left( \Hei\right) }$ the closure of $H\left( \Hei\right)$ with respect to the norm $\max_{0 \leqslant t \leqslant 1} \rho (\gamma (t) ), \; \gamma \in \WH$, and for $\rho$ a homogeneous norm on $\Hei$. By Remark \ref{r.norms}, $\overline{H\left( \Hei\right) }$ is independent of $\rho$.

\section{Proof of Theorem \ref{Thm.Support}}\label{sec3}
We will divide the proof of Theorem \ref{Thm.Support} in two steps. First, we introduce a family of processes that approximates $\{g_t\}_{0\leqslant t\leqslant 1}$. This is used in Corollary \ref{c.weak.convergence}  to show that the support $\mathcal{S}_{\mu}$ is contained in $\overline{H\left( \Hei\right) }$. The reverse inclusion is proven in Corollary \ref{c.reverse.inclusion}  which follows from Theorem \ref{thm.estimate}. In Proposition \ref{p.closure } we prove that $\overline{H\left( \Hei\right) }= \WH$, which concludes the proof of  Theorem \ref{Thm.Support}. 

\subsection{Approximation of the hypoelliptic Brownian motion} 
The aim of this step is to show that the support  $\mathcal{S}_\mu$ of the law of $\{g_t\}_{0\leqslant t \leqslant 1}$ is contained in $\overline{H\left( \Hei\right) }$. This will be accomplished by constructing a horizontal piecewise approximation $g_\delta(t)$ of $g_t$ such that $\mu_{\delta} \rightarrow \mu $ weakly, where $\mu_{\delta}$ is the law of $g_\delta(t)$. Different approximations of a Brownian motion have been extensively studied over the decades, see for example Wong-Zakai \cite{WongZakai1965}, Kunita \cite{Kunita1974},  Nakao-Yamamoto \cite{NakaoYamato1978}, Ikeda-Nakao-Yamato \cite{IkedaNakaoYamato1977},  and Ikeda-Watanabe \cite[Chapter 6, Section 7]{IkedaWatanabe1989} for more details. We are not able to refer to all the vast literature on the subject, but we mentioned some results which are closer and more relevant to the techniques we use in this paper. 

Let $\{ B_{\delta} \}_{\delta >0}$ be an approximation of a two-dimensional Brownian motion, that is, 
\begin{equation}\label{convergence.BM}
\E \left[ \max_{0 \leqslant t \leqslant 1} \vert B_{\delta} (t)-B_{t} \vert_{\R^2}^2 \right] \longrightarrow 0, \quad \text{as} \; \, \delta \to 0, 
\end{equation}
such that 
\begin{equation}\label{convergence.LevyArea}
\E \left[ \max_{0 \leqslant t \leqslant 1}  A_{\delta} (t)-A_{t} \vert_{\R}^2 \right] \longrightarrow 0 \quad \text{as} \; \, \delta \to 0,
\end{equation}
where 
\begin{equation}\label{approximation.LevyArea}
A_\delta(t) := \frac{1}{2} \int_0^t \left( B_{1,\delta} (s)B_{2,\delta}^\prime(s) -  B_{2,\delta} (s)B_{1,\delta}^\prime(s) \right)ds.
\end{equation}
Let $f_{1}$ and $f_{2}$ be differentiable functions on $[0,1]$ such that $f_i(0)=0$ and $f_i(1)=1$ for $i=1,2$. Set
\begin{equation}\label{approximationBM}
B_{i,\delta} (t):= B_i (k \delta ) + f_i \left( \frac{t-k\delta}{\delta} \right)  \left( B_i (k\delta +\delta ) - B_i(k\delta)   \right), \;  k\delta \leqslant t < (k+1)\delta,
\end{equation}
then by \cite[Theorem 7.1]{IkedaWatanabe1989} the family $\{ B_{\delta} \}_{\delta >0}$ satisfies \eqref{convergence.BM} and \eqref{convergence.LevyArea}.  Let us define now a sequence of  processes $g_\delta (t)$ on $\Hei$.

\begin{definition}
Let $B_\delta (t)$ be an approximation of a two-dimensional Brownian motion satisfying \eqref{convergence.LevyArea}. For each $\delta$, $t$, and $\omega$ we set 
\[
g_\delta (t) = \left( g_{1,\delta} (t) , g_{2,\delta} (t) , g_{3,\delta} (t)  \right),
\]
where
\begin{align}\label{approximation.Hypo.BM}
& g_{1,\delta} (t) = B_{1,\delta}(t) \notag
\\
& g_{2,\delta} (t) = B_{2,\delta}(t)
\\
& g_{3,\delta} (t) = A_\delta (t). \notag 
\end{align}
\end{definition}

Let $C^2_p \left( \R^2 \right)$ be the space of piecewise continuously twice differentiable curves in $\R^2$ starting at zero, and set
\[
H_p \left( \Hei \right): = \left\{ \gamma : [0,1] \longrightarrow \Hei, \,\pi (\gamma) \in C^{2}_{p}(\R^2), \, 
\gamma_3(t) = \frac{1}{2} \int_0^t \omega \left( \pi(\gamma) (s),\pi(\gamma)^\prime(s) \right)ds  \right\},
\]
where $\gamma =  \left( \pi (\gamma) , \gamma_3 \right)$, that is, $H_p \left( \Hei \right)$ is the set of piecewise continuously twice  differentiable horizontal curves. Clearly we have that 
\[
\overline{H\left( \Hei\right) } = \overline{H_p\left( \Hei\right) }.
\]
We can view $g_\delta$ as a $H_p\left( \Hei \right)$-valued random variable, that is,
\begin{align}\label{e.map}
g_\delta : \Omega \longrightarrow H_p \left( \Hei \right), & \quad    \omega \mapsto \left\{ t \mapsto g_\delta(t,\omega) \right\},
\end{align}
and hence we can induce a probability measure $\mu_\delta$ on $\WH$ by 

\[
\mu_\delta (E) := \Prob \left( g_\delta ^{-1} \left(E \cap H_p \left( \Hei \right) \right) \right) 
\]
for any Borel set $E$ in $\WH$.

\begin{proposition}\label{p.support.approximation}
Let $\mathcal{S}_{\mu_\delta}$ be the support of the measure $\mu_\delta$. Then 
\[
\mathcal{S}_{\mu_{\delta}} \subset \overline{H_p\left( \Hei\right) }=\overline{H\left( \Hei\right) }.
\]
\end{proposition}

\begin{proof}
By \ref{e.map} we have that $ g_\delta (\Omega ) \subset  H_p\left( \Hei\right) $ and hence 

\[
\Omega \subset g_\delta^{-1} g_\delta \left( \Omega \right) \subset g_\delta ^{-1} \left( H_p\left( \Hei\right) \right) \subset \Omega.
\]
Therefore by the definition of $\mu_\delta$ it follows that
\[
1 = \Prob \left(  g_\delta ^{-1} \left( H_p\left( \Hei\right) \right) \right) = \mu_\delta\left( H_p\left( \Hei\right) \right) \leqslant  \mu_\delta \left( \overline{H_p\left( \Hei\right) }\right) \leqslant 1,
\]
and the proof is complete since $S_{\mu_\delta}$ is the smallest closed subset of $\WH$ having $\mu_\delta$-measure one.
\end{proof}

We can now state and prove the main result of this section, that is, that the family $\left\{ g_\delta \right\}_{\delta>0} $ is an approximation of the hypoelliptic Brownian motion $g$ in the sense that $ \E \left[ \max_{0 \leqslant t \leqslant 1} \rho ( g_{\delta} (t)^{-1} g_{t} )^{2} \right] \longrightarrow 0$ as $ \delta \to 0$ for any homogeneous norm $\rho$. As a consequence, the support of the measure $\mu$ is contained  in $\overline{H\left( \Hei\right) }$.

\begin{theorem}\label{thm.convergence.Hypo.BM}
Let $\left\{ g_\delta \right\}_{\delta >0}$ be the sequence  defined  by \eqref{approximation.Hypo.BM}, and $\rho$ be a homogeneous norm on $\Hei$. Then

\begin{equation}\label{e.convergence.Hypo.BM}
\lim_{\delta \rightarrow 0} \E \left[ \max_{0 \leqslant t \leqslant 1} \rho ( g_{\delta} (t)^{-1} g_{t} )^2 \right]=0.
\end{equation}
\end{theorem}
\begin{proof} 

By \eqref{e.DistEquivalence} and \eqref{hom.norm} we have that 

\begin{align*}
&\max_{0 \leqslant t \leqslant 1} \rho ( g_{\delta} (t)^{-1} g_{t} )^4 \leqslant C \max_{0\leqslant t\leqslant 1} \vert g_{\delta}^{-1}(t) g_{t} \vert^{4}
\\
& \leqslant   C \max_{0\leqslant t\leqslant 1} \vert  B_{t}- B_{\delta}(t)  \vert^{4} + C \max_{0\leqslant t\leqslant 1} \left| A_{t} -A_{\delta} (t) - \frac{1}{2} \omega \left( B_{\delta} (t), B_{t} \right) \right|^{2}
\\
&\leqslant  C \left( \max_{0\leqslant t\leqslant 1} \vert B_{t}- B_{\delta} (t) \vert^{2} + \max_{0\leqslant t\leqslant 1} \left| A_{t} -A_{\delta} (t) - \frac{1}{2} \omega \left( B_{\delta} (t), B_{t} \right) \right| \right)^2,
\end{align*}
and hence
\begin{align*}
&  \E\left[ \max_{0 \leqslant t \leqslant 1} \rho ( g_{\delta} (t)^{-1} g_{t} )^2 \right] 
\\
& \leqslant  C \E \left[ \max_{0\leqslant t\leqslant 1} \vert B_{t}- B_{\delta} (t) \vert^{2}  \right] + C \E \left[\max_{0\leqslant t\leqslant 1} \left| A_{t} -A_{\delta} (t) \right|\right] + C \E \left[ \max_{0\leqslant t\leqslant 1} \left| \frac{1}{2} \omega \left( B_{\delta} (t), B_{t} \right) \right| \right]
\\
& \leqslant C \E \left[ \max_{0\leqslant t\leqslant 1} \vert B_{t}- B_{\delta} (t) \vert^{2}  \right] + C \E \left[\max_{0\leqslant t\leqslant 1} \left| A_{t} -A_{\delta} (t) \right|^{2}\right]^{\frac{1}{2}} + C \E \left[ \max_{0\leqslant t\leqslant 1} \left| \frac{1}{2} \omega \left( B_{\delta} (t), B_{t} \right) \right| \right],
\end{align*}
for some constant $C$ (which varies from line to line). By \eqref{convergence.BM} and \eqref{convergence.LevyArea}, we only need to show that 
\[
\E \left[ \max_{0\leqslant t\leqslant 1} \left| \frac{1}{2} \omega \left( B_{\delta} (t), B_{t} \right) \right| \right] \longrightarrow 0, \text{ as $\delta \mapsto 0 $. }
\]
Since $B_{i}$ is independent of $ B_{j,\delta}  - B_{j}$ when $i \neq j$, and
\begin{align*}
& \max_{0\leqslant t\leqslant 1} \left| \frac{1}{2} \omega \left( B_{\delta} (t), B_{t} \right) \right| \leqslant \frac{1}{2}  \max_{0\leqslant t\leqslant 1} \vert B_{1} (t) \vert   \max_{0\leqslant t\leqslant 1}  \vert B_{2,\delta} (t) - B_{2} (t) \vert 
\\
& +\frac{1}{2}  \max_{0\leqslant t\leqslant 1} \vert B_{2} (t) \vert   \max_{0\leqslant t\leqslant 1}  \vert B_{1,\delta} (t) - B_{1} (t) \vert ,
\end{align*}
the proof is complete.
\end{proof}

\begin{corollary}\label{c.weak.convergence}
We have that $\mu_{\delta} \rightarrow \mu$ weakly. In particular
\begin{equation}\label{support.Hypo.measure}
\mathcal{S}_{\mu} \subset \overline{H\left( \Hei\right) }.
\end{equation}
\end{corollary}

\begin{proof}
Let us first show that $\{g_{\delta} \}_{\delta >0}$ converges to $g$ in probability in $\WH$. For any fixed  $\varepsilon>0$ we have that 
\[
\Prob \left( \max_{0\leqslant t \leqslant 1} \rho (g_{\delta} (t)^{-1} g_{t} ) > \varepsilon \right)\leqslant \frac{1}{\varepsilon^2} \E \left[  \max_{0\leqslant t \leqslant 1} \rho (g_{\delta} (t)^{-1} g_{t} )^2 \right] 
\] 
which goes to zero by Theorem \ref{thm.convergence.Hypo.BM}. Therefore $\{g_{\delta} \}_{\delta >0}$ converges to $g$ in distribution, and hence $\mu_\delta$ converges weakly to $\mu$ in $\WH$. Thus, for any closed set $F$ in $\WH$ we have that 
\[
 \mu (F) \geqslant \limsup_{\delta \rightarrow 0} \mu_\delta (F) .
\]
In particular, for $F= \overline{H_p\left( \Hei\right) }$ and by Proposition \ref{p.support.approximation} it follows that 
\[
\mu \left( \overline{H_p\left( \Hei\right) } \right) \geqslant  \limsup_{\delta \rightarrow 0} \mu_\delta (\overline{H_p\left( \Hei\right) }) =1.
\]
Since $\mathcal{S}_\mu$ is the smallest closed subset having $\mu$-measure one, we have that $ \mathcal{S}_\mu \subset \overline{H_p\left( \Hei\right) }  = \overline{H\left( \Hei\right) }$.

\end{proof}

We conclude this section showing that for each fixed $\delta$, the measures $\mu$ and $\mu_\delta$ are singular.
\begin{proposition}\label{p.singular}
For each $\delta$ the measures $\mu$ and $\mu_\delta$ are singular.
\end{proposition}
\begin{proof}
From the proof of Proposition \ref{p.support.approximation} we know that $\mu_\delta \left( H\left( \Hei \right) \right)=1 $. It is then enough to show that $\mu \left( H\left( \Hei \right) \right)=0 $. Let us denote by $\nu$ the law of a two-dimensional standard Brownian motion. By definition of $g_t$ and $\pi$ \eqref{eqn.projection}, the following diagram commutes 
 \[
  \begin{tikzcd}
    \Omega \arrow{r}{g} \arrow[swap]{dr}{B} & \WH \arrow{d}{\pi} \\
     & \WR,
  \end{tikzcd}
  \]
and for any Borel set $E$ in $\WR$ we have that
\[
\nu (E) := \Prob \left( B^{-1} (E) \right) = \Prob \left( g^{-1} \circ \pi ^{-1} (E) \right) = \mu \left( \pi^{-1}(E) \right).
\]
Moreover, from Remark  \ref{r.explanation.name} we know that $\pi\left( H \left( \Hei \right) \right) $ is  the Cameron-Martin space over $\R^2$, which is known to have $\nu$-measure zero, see \cite{Gross1967} for more details. Therefore we can conclude that 
\[
\mu\left( H \left( \Hei \right) \right) \leqslant \mu \left(\pi^{-1} \pi\left( H \left( \Hei \right) \right) \right) = \nu \left(  \pi\left( H \left( \Hei \right) \right) \right)=  0.
\]

\end{proof}

\subsection{Support of the Wiener measure}

The goal of this section is to prove that  $\overline{H\left( \Hei\right) } \subset \mathcal{S}_{\mu}$ which will follow from Theorem \ref{thm.estimate}. Moreover, in Proposition  \ref{p.closure } we show that  $\overline{H\left( \Hei\right) } = \WH$.

\begin{theorem}\label{thm.estimate}
Let $\phi = \left( \pi (\phi), \phi_3 \right) = \in H_p\left( \Hei\right)$. For $\delta>0$ let us denote by $E_{\delta, \phi}$ the event 
\[
E_{\delta, \phi} := \left\{ \max_{0\leqslant t \leqslant 1} \vert B_t - \pi (\phi) (t) \vert_{\R^2} < \delta \right\}.
\]
Then for any $\varepsilon>0$
\[
 \lim_{\delta \rightarrow 0}  \Prob \left(   \max_{0\leqslant t \leqslant 1} \rho (\phi(t)^{-1} g_{t} )> \varepsilon \, \vert \,E_{\delta, \phi}  \right)=0,
 \]
where $\rho$ is a homogeneous norm on $\Hei$. 
\end{theorem}

\begin{proof}

By \eqref{e.DistEquivalence} it is enough to prove it for the homogeneous norm given by \eqref{hom.norm}. For $\phi\in H_p\left( \Hei\right)$ we have that
\begin{align*}
& \max_{0\leqslant t \leqslant 1} \vert \phi (t)^{-1} g_t \vert^4\leqslant \max_{0\leqslant t \leqslant 1} \vert B_t-\pi(\phi) (t)\vert^4_{\R^2} 
\\
&+ \max_{0\leqslant t \leqslant 1} \left| \frac{1}{2} \int_0^t  \omega\left( B_s-\pi (\phi)(s), dB_s -\pi (\phi)^{\prime} (s)ds \right)+ \int_0^t  \omega\left( B_s-\pi (\phi)(s), \pi(\phi)^\prime (s)\right)ds \right| ^2 
\\
& \leqslant  \left(  \max_{0\leqslant t \leqslant 1} \vert B_t-\pi(\phi) (t)\vert^{2}_{\R^2}  \right.
\\
&\left. + \max_{0\leqslant t \leqslant 1} \left| \frac{1}{2} \int_{0}^{t}  \omega\left( B_{s}-\pi (\phi)(s), dB_{s} -\pi(\phi)^{\prime} (s)ds \right)+ \int_{0}^{t}  \omega\left( B_s-\pi(\phi)(s), \pi(\phi)^{\prime}(s) \right)ds \right| \right)^2.
\end{align*}
Therefore on the event $E_{\delta, \phi}$ we have that 
\begin{align*}
&  \max_{0\leqslant t \leqslant 1} \vert \phi (t)^{-1} g_t \vert^2 \leqslant  \max_{0\leqslant t \leqslant 1} \vert B_t-\pi (\phi) (t)\vert^{2}_{\R^2} 
\\
& +  \max_{0\leqslant t \leqslant 1} \left| \frac{1}{2} \int_0^t  \omega\left( B_s-\pi(\phi)(s), dB_{s} -\pi (\phi)^\prime (s)ds \right)+ \int_0^t  \omega\left( B_s-\pi(\phi)(s), \pi(\phi)^\prime(s) \right)ds \right| 
\\
& \leqslant  \delta^2 + \max_{0\leqslant t \leqslant 1} \left| \int_0^t  \omega\left( B_s-\pi(\phi)(s), \pi(\phi)^{\prime}(s) \right)ds \right|  + \max_{0\leqslant t \leqslant 1} \left| \frac{1}{2} \int_0^t  \omega\left( B_s-\pi(\phi)(s), dB_s -\pi(\phi)^{\prime} (s)ds \right) \right|  
\\
& \leqslant  \delta^2 + \delta C_\phi  +  \max_{0\leqslant t \leqslant 1} \left| \frac{1}{2} \int_0^t  \omega\left( B_s-\pi(\phi)(s), dB_s -\pi (\phi)^{\prime} (s)ds \right) \right|,
\end{align*}
where $ C_\phi:= \int_0^1 \vert \phi_1^\prime (s) \vert +\vert \phi_2^\prime (s) \vert ds$. It then follows that

\begin{align*}
& \Prob \left(  \max_{0\leqslant t \leqslant 1} \vert \phi (t)^{-1} g_t \vert > \varepsilon \, \vert E_{\delta, \phi} \right) 
\\
& \leqslant \Prob \left( \max_{0\leqslant t \leqslant 1} \left| \frac{1}{2} \int_0^t  \omega\left( B_s-\pi (\phi)(s), dB_s -\pi (\phi)^{\prime} (s)ds \right) \right|  > \varepsilon^2 -\delta C_\phi - \delta^2 \;  \;  \vert \; \; E_{\delta, \phi} \right).
\end{align*}
Note that this last expression only depends on the process $ B_t - \pi (\phi)(t)$. Since $\phi = \left( \pi (\phi), \phi_3 \right) \in H_p\left( \Hei \right) $, by Remark  \ref{r.explanation.name} we know that $\pi (\phi )$ belongs to the Cameron-Martin space over $\R^{2}$.  Therefore from the Cameron-Martin-Girsanov Theorem there exists a probability measure $\Q^\phi$ such that the process $B^\phi_t:=B_t +\pi (\phi) (t)$ is a Brownian motion under $\Q^{\phi}$. More precisely there exists an exponential martingale $\mathcal{E}^\phi$ such that 
\[
\Q^\phi (A) = \E \left[ \mathcal{E}^\phi \mathbbm{1}_A  \right] \quad \forall A \in \mathcal{F},
\]
where $\mathcal{E}^\phi= \exp \left( -\int_0^1 \langle \pi(\phi)^{\prime} (s), dB_{s} \rangle_{\R^2} ds  - \frac{1}{2} \int_0^1 \vert \pi (\phi)^{\prime} (s) \vert^{2}_{\R^2} ds \right)$.  Note that 
\begin{align*}
& d \left( B_{t} - \pi ( \phi ) (t) \right) = dB_t - \pi( \phi )^{\prime} (t) dt, \; \text{and}
\\
& dB_{t} = dB^{\phi}_{t} - \pi (\phi)^{\prime}(t) dt,
\end{align*}
that is, the law of $ B_t - \pi(\phi) (t) $ under $\Prob$ is the same as the law of $B_t$ under $\Q^\phi$. Therefore we can write 
\begin{align*}
& \Prob\left( E_{\delta, \phi} \right) = \Prob \left(  \max_{0\leqslant t \leqslant 1} \vert B_t - \pi( \phi) (t) \vert_{\R^2} < \delta \right)  
\\
&= \Q^\phi \left(\max_{0\leqslant t \leqslant 1} \vert B_t  \vert_{\R^2} < \delta \right) = \E \left[ \mathcal{E}^\phi \mathbbm{1}_{E_\delta} \right] =\E \left[ \mathcal{E}^\phi \vert  E_\delta \right] \Prob \left( E_\delta \right)  ,
\end{align*}
where we set $ E_\delta:= \left\{ \max_{0\leqslant t \leqslant 1} \vert B_t\vert_{\R^2} < \delta \right\}$.
Similarly we have that
\begin{align*}
&\Prob \left( \max_{0\leqslant t \leqslant 1} \left| \frac{1}{2} \int_0^t  \omega\left( B_s-\pi (\phi)(s), dB_{s} -\pi(\phi)^{\prime} (s)ds \right) \right|  > \varepsilon^2 -\delta C_\phi - \delta^2 \, , \; E_{\delta, \phi} \right)  
\\
& =  \E \left[ \mathcal{E}^\phi \vert F_{\delta,  \phi }^\varepsilon \cap E_{\delta} \right] \Prob \left( F_{\delta, \phi}^\varepsilon \cap E_{\delta} \right),
\end{align*}
where $F_{\delta, \phi}^\varepsilon := \left\{ \max_{0\leqslant t \leqslant 1} \left| \frac{1}{2} \int_0^t  \omega\left( B_s, dB_s \right) \right|  > \varepsilon^2 -\delta C_\phi - \delta^2 \right\}  $. Therefore it follows that 
\begin{align}\label{e.idk4}
&\Prob \left(  \max_{0\leqslant t \leqslant 1} \vert \phi (t)^{-1} g_t \vert > \varepsilon \, \vert E_{\delta, \phi} \right)  \nonumber 
\\
& \leqslant \Prob \left( \max_{0\leqslant t \leqslant 1} \left| \frac{1}{2} \int_0^t  \omega\left( B_s-\pi (\phi)(s), dB_s -\pi (\phi )^{\prime} (s)ds \right) \right|  > \varepsilon^2 -\delta C_\phi - \delta^2 \, \vert E_{\delta, \phi} \right)  \nonumber
\\
&  = \frac{  \Prob \left( F_{\delta, \phi}^\varepsilon \cap E_{\delta} \right)  E \left[ \mathcal{E}^\phi \vert F_{\delta,  \phi }^\varepsilon \cap E_{\delta} \right] }{ \Prob \left( E_\delta \right)  \E \left[ \mathcal{E}^\phi \vert  E_\delta \right]  }   =  \Prob \left( F_{\delta, \phi}^\varepsilon \, \vert \, E_{\delta} \right)\times \frac{ E \left[ \mathcal{E}^\phi \vert F_{\delta,  \phi }^\varepsilon \cap E_{\delta} \right]  }{ \E \left[ \mathcal{E}^\phi \vert  E_\delta \right]  }   
\end{align}
We will show later in the paper, see Lemma \ref{l.claim}, that for any $\varepsilon >0$ and any $\phi \in H_p \left( \Hei \right)$ we have that 
\begin{equation}\label{e.claim}
\lim_{\delta \rightarrow 0} \frac{\E \left[ \mathcal{E}^\phi \, \vert \, F_{\delta, \phi}^\varepsilon \cap E_{\delta}  \right]}{\E \left[ \mathcal{E}^\phi \, \vert \,   E_{\delta}\right]}=1.
\end{equation}

In light of \ref{e.idk4} and \ref{e.claim}, the proof will be completed once we show that 
\begin{align*}
&\lim_{\delta \rightarrow 0} \Prob \left(F_{\delta, \phi}^\varepsilon  \, \vert \, E_\delta \right) :=
\\
&\lim_{\delta \rightarrow 0}\Prob \left(   \max_{0\leqslant t \leqslant 1} \left| \frac{1}{2} \int_0^t  \omega\left( B_s, dB_s \right) \right|  > \varepsilon^2 -\delta C_\phi - \delta^2\; \left| \; \; \max_{0\leqslant t \leqslant 1} \vert B_t\vert_{\R^2} < \delta  \right)  =0. \right.
\end{align*}

The process $ A_t:=\frac{1}{2} \int_0^t \omega\left( B_s, dB_s \right)$ is a square integrable martingale with zero mean, and therefore there exists a one dimensional Brownian motion $b_t$  such that
\[
b_{\tau(t)}  = \frac{1}{2} \int_0^t  \omega\left( B_s, dB_s \right),
\]
where  $\tau(t)= \frac{1}{4} \int_0^t  B_1(s) ^2+  B_2(s)^2ds$. Moreover it is known that $b_t$ is independent of $B_{t}$ \cite[Chapter 6 p. 470]{IkedaWatanabe1989}. Hence we have that
\begin{align*}
&\Prob \left( \max_{0\leqslant t \leqslant 1} \vert b_{\tau(t)} \vert  >\varepsilon^2 -\delta C_\phi - \delta^2 \vert , \;  \max_{0\leqslant t \leqslant 1} \vert B_t\vert_{\R^2} < \delta \right)
\\
& \leqslant  \Prob \left( \max_{0 \leqslant t \leqslant \frac{1}{4} \delta^2 } \vert b_{t} \vert  >\varepsilon^2 -\delta C_\phi - \delta^2 \, \vert , \; \max_{0\leqslant t \leqslant 1} \vert B_t\vert_{\R^2} < \delta  \right) 
\\
& = \Prob \left( \max_{0 \leqslant t \leqslant \frac{1}{4} \delta^2} \vert b_{t} \vert  >\varepsilon^2 -\delta C_\phi - \delta^2  \right) =  \Prob \left( \max_{0\leqslant t \leqslant 1} \vert b_{t} \vert  >2 (\frac{\varepsilon^2}{\delta}- C_\phi-\delta)  \right),
\end{align*}
which goes to zero as $\delta$ goes to zero.
\end{proof}

\begin{corollary} \label{c.reverse.inclusion}
$\overline{H\left( \Hei\right) } \subset \mathcal{S}_{\mu}. $

\end{corollary}

\begin{proof}
Let us first prove that for any $\phi \in H_p \left( \Hei\right) $ and $\varepsilon >0$ we have that  $ \mu \left( B_\varepsilon (\phi ) \right) >0$, where
\[
B_\varepsilon (\phi ):= \left\{ \gamma \in \WH, \; \max_{0 \leqslant t \leqslant 1} \vert \phi^{-1} (t) \gamma (t) \vert <\varepsilon  \right\}
\]
Indeed, for any $\phi \in H_p \left( \Hei\right) $ and $\varepsilon >0$ we have  that

\begin{align*}
&\mu \left( B_\varepsilon (\phi ) \right):= \Prob \left( g \in B_\varepsilon (\phi ) \right) = \Prob \left( \max_{0 \leqslant t \leqslant 1} \vert \phi (t)^{-1} g_{t} \vert < \varepsilon \right) 
\\
& \geqslant \Prob \left( \max_{0 \leqslant t \leqslant 1} \vert \phi (t)^{-1} g_{t} \vert  < \varepsilon \, \vert \,E_{\delta, \phi} \right) \Prob \left( E_{\delta, \phi }\right),
\end{align*}
where $E_{\delta, \phi} := \left\{ \max_{0\leqslant t \leqslant 1} \vert B_t - \pi ( \phi) (t) \vert_{\R^2} < \delta \right\}.$ From Theorem \eqref{thm.estimate} there exists a $\delta_0$ such that for every $\delta\in (0,\delta_0)$
\[ 
\Prob \left( \max_{0 \leqslant t \leqslant 1} \vert \phi (t)^{-1} g_{t} \vert  < \varepsilon \, \vert \, E_{\delta, \phi} \right)  \geqslant \frac{1}{2},
\]
for any $\varepsilon>0$. Combining everything together we have that
\[
\mu \left( B_\varepsilon (\phi ) \right) \geqslant \frac{1}{2} \Prob \left( \sup_{0\leqslant t \leqslant 1} \vert B_t - \pi ( \phi ) (t) \vert_{\R^2}  < \frac{\delta_0}{2} \right) ,
\]
and the latter is positive since $\pi (\phi )$ is in the Cameron-Martin space over $\R^{2}$.
Therefore, if $O$ is any open set in $\WH$ with $\mu (O)=0$ then $O \subset H_p \left( \Hei \right) ^c $, and hence 
\begin{align*}
& \bigcup_{ \mathclap{ \substack{     O \, \text{open}  \\ \mu(O)=0          }}} O \subset    H_p \left( \Hei \right) ^c , \; \; \text{that is, } \; \; \mathcal{S}_{\mu} := \bigcap_{ \mathclap{ \substack{     F \, \text{closed}  \\ \mu(F)=1          }} } F \supset  H_p \left( \Hei \right) , 
\end{align*}
and since $S_\mu$ is closed, we have that   $\mathcal{S}_{\mu} \supset \overline{H_p\left( \Hei\right) }=\overline{H\left( \Hei\right) } $.
\end{proof}

The proof of Theorem \ref{thm.estimate} will be completed once we show \eqref{e.claim}. Before proceeding to the proof of \eqref{e.claim}, we need the following lemma whose proof can be found in \cite[ pp. 536-537]{IkedaWatanabe1989}.

\begin{lemma}[pp. 536-537 in \cite{IkedaWatanabe1989}]\label{Lemma3.1}
Let $I_1, \ldots, I_n$ be $n$ random variables on a probability space $\left( \Omega, \mathcal{F}, \Prob \right)$. Let $\left\{  A_\delta \right\}_{0<\delta <1} $ be a family of events in $\mathcal{F}$ and $a_1, \ldots , a_n$ be $n$ numbers. If for every real number $c$ and every $1\leqslant i \leqslant n$
\[
\limsup_{\delta \rightarrow 0} \E \left[ \exp (c\, I_i)\, |  A_\delta   \right] \leqslant \exp (c\,a_i),
\]
then
\[
\lim_{\delta\rightarrow 0} \E \left[ \exp\left(\sum_{i=1}^n I_i\right)  |  A_\delta  \right] = \exp \left( \sum_{i=1}^n a_i \right).
\]
\end{lemma}

\begin{lemma}\label{l.claim}
Let $E_\delta$ and $F^\varepsilon_{\delta, \phi} $ be given as in the proof of Theorem \ref{thm.estimate}.  Then 
\[
\lim_{\delta \rightarrow 0} \frac{\E \left[ \mathcal{E}^\phi \, \vert \, F_{\delta, \phi}^\varepsilon \cap E_{\delta}  \right]}{\E \left[ \mathcal{E}^\phi \, \vert \,   E_{\delta}\right]}=1.
\]
\end{lemma}

\begin{proof}
Let us first prove that 
\begin{equation}\label{e.idk3}
 \lim_{\delta \rightarrow 0}  \E \left[ \mathcal{E}^\phi \, \vert \, E_\delta \right] = \exp\left(   - \frac{1}{2} \int_0^1 \vert \pi (\phi)^{\prime} (s) \vert^2_{\R^2} ds\right).
\end{equation}
Since $\mathcal{E}^{\phi}= \exp \left( -\int_0^1 \langle \pi (\phi)^{\prime} (s), dB_{s} \rangle_{\R^2} ds  - \frac{1}{2} \int_0^1 \vert \pi(\phi)^{\prime} (s) \vert^2_{\R^2} ds \right)$, by Lemma \ref{Lemma3.1} and the definition of $E_\delta$, it is enough to show that for any real number $c$ and $i=1,2$ 
\[
\limsup_{\delta \rightarrow 0} \E \left[ \exp\left( - c\int_0^1 \phi^\prime_i (s) dB_i(s) \right) \, \left| \, \max_{0
\leqslant t \leqslant 1} \vert B_t \vert_{\R^2} < \delta \right]  \leqslant 1.  \right.
\]  
For $\phi \in H_p \left( \Hei \right)$ we can write $\int_0^1 \phi^\prime_i (s) dB_i(s) = \phi^\prime_i (1) B_i(1) - \int_0^1 \phi^{\prime \prime}_i  (s) B_i(s)ds$, and hence on the event $E_\delta$ we have that 
\begin{align*}
& \exp\left( - c\int_0^1 \phi^\prime_i (s) dB_i(s) \right)  \leqslant \exp\left( - ck_\phi \delta \right),
\end{align*}
for some finite constant $k_\phi$ only depending on $\phi$. Therefore we have that 
\begin{align*}
& \limsup_{\delta \rightarrow 0} \E \left[ \exp\left( - c\int_0^1 \phi^\prime_i (s) dB_i(s) \right) \, \left| \, \max_{0
\leqslant t \leqslant 1} \vert B_t \vert_{\R^2} < \delta \right] \leqslant \E \left[ \limsup_{\delta \rightarrow 0}  e^{- ck_\phi \delta}  \vert E_\delta \right]  \leqslant 1.  \right.
\end{align*}
In a similar way it can be shown that 
\[
\lim_{\delta \rightarrow 0}  \E \left[ \mathcal{E}^\phi \, \vert \,  F_{\delta, \phi}^\varepsilon \cap E_{\delta} \right] = \exp\left(   - \frac{1}{2} \int_0^1 \vert \pi (\phi)^{\prime} (s) \vert^2_{\R^2} ds\right),
\]
and the proof is completed.
\end{proof}
The following Proposition completes the proof of Theorem \ref{Thm.Support}.
\begin{proposition}\label{p.closure }
We have that $\overline{H\left( \Hei\right) }= \WH$.
\end{proposition}
\begin{proof}
Any element in $\WH$ can be approximated with piecewise linear curves in the  uniform topology. It is then enough to prove that for any piecewise linear curve $\xi$ there exists a sequence of horizontal finite energy curves $\left\{ \phi^\xi_n \right\}_{n\in \mathbb{N}}$ such that $ d_{\WH} \left( \phi^\xi_n,\xi \right) \rightarrow 0 $. Let us first explain the geometric construction through the following example. Consider the curve  $ t \rightarrow \xi (t) = (0,0,t) \in \Hei$ for $t\in[0,1]$, which is the prototype of a non-horizontal curve. Let us define a family of finite energy horizontal curves $\phi_n$ by 
\[
\phi_n (s) := \left( \frac{2}{n} \cos \left( n^2 s \right), \frac{1}{n} \sin \left( n^2 s \right), s \right).
\]
Geometrically, the curves $\phi_n$ are helics that shrink around the $\xi$ as $n$ goes to infinity. Indeed, 
\begin{align*}
& d_{\WH} \left( \phi_n, \xi \right) ^4 = \max_{0\leqslant s \leqslant 1 } \left[ \left( \left( \frac{2}{n} \cos \left( n^2 s \right)\right)^2 + \left(\frac{1}{n} \sin \left( n^2 s \right) \right)^2 \right)^2   \right.
\\
&\left. + \left(  s- \frac{1}{2} \int_0^t  \omega \left( \pi (\phi)_{n} (u), \pi (\phi)^{\prime}_{n} (u) \right)du \right)^2    \right]
\\
& = \max_{0\leqslant s \leqslant 1 }  \left( \left( \frac{2}{n} \cos \left( n^2 s \right)\right)^2 + \left(\frac{1}{n} \sin \left( n^2 s \right) \right)^2 \right)^2  \longrightarrow 0,
\end{align*}
as $n$ goes to infinity.  Now, let $\xi (t) = ( a_1 t, a_2 t, a_3 t)$ be a linear curve in $\Hei$, where $a_1, \, a_2, \, a_3 \in \R$. Then set
\begin{align*}
 &\phi_n (s) := \left( a_1 s + \frac{2}{n} \cos\left( n^2 a_3 s \right), a_2 s + \frac{1}{n} \sin\left( n^2 a_3 s \right), \right.
\\
& \left.  a_3 s - \frac{a_2s}{n} \cos \left( n^2 a_3 s \right) + \frac{a_1s}{2n} \sin \left( n^2a_3 s\right) + \frac{1}{n} \int_0^s  2a_2 \cos\left(n^2a_3u\right) - a_1 \sin\left(n^2a_3u\right) du \right).
\end{align*}
It is easy to check that for any $ n \in \mathbb{N}$, $\phi_n$ is a finite energy horizontal curve such that 
\begin{align*}
& (\phi_n^{-1} \xi )(s) 
\\
& = \left( -\frac{2}{n} \cos\left( n^2 a_3 s \right) , - \frac{1}{n} \sin \left( n^2 a_3 s \right),   \frac{1}{n} \int_0^s a_1 \sin \left( n^2a_3u\right) - 2a_2 \cos \left( n^2a_3 u \right) du \right),
\end{align*}
which implies that $d_{\WH}  \left(\phi_n , \xi \right) \rightarrow 0 $ as $n$ goes to infinity.
\end{proof}

\begin{acknowledgement}
The author wishes to thank M. Gordina and an anonymous referee for carefully reading the manuscript and suggesting significant improvements.
\end{acknowledgement}

\providecommand{\bysame}{\leavevmode\hbox to3em{\hrulefill}\thinspace}
\providecommand{\MR}{\relax\ifhmode\unskip\space\fi MR }
\providecommand{\MRhref}[2]{%
  \href{http://www.ams.org/mathscinet-getitem?mr=#1}{#2}
}
\providecommand{\href}[2]{#2}


\begin{thebibliography}{10}

\bibitem{Aida1990}
Shigeki Aida, \emph{Support theorem for diffusion processes on {H}ilbert
  spaces}, Publ. Res. Inst. Math. Sci. \textbf{26} (1990), no.~6, 947--965.
  \MR{1079903}

\bibitem{BonfiglioliLanconelliUguzzoniBook}
A.~Bonfiglioli, E.~Lanconelli, and F.~Uguzzoni, \emph{Stratified {L}ie groups
  and potential theory for their sub-{L}aplacians}, Springer Monographs in
  Mathematics, Springer, Berlin, 2007. \MR{2363343}
  
\bibitem{FrizVictoirBook}
P.~Friz, and N.~Victoir,
\emph{Multidimensional Stochastic Processes as Rough Paths}, Cambridge studies in advanced mathematics, Cambridge, 2010. 

\bibitem{Gross1967}
Leonard Gross, \emph{Abstract {W}iener spaces}, Proc. {F}ifth {B}erkeley
  {S}ympos. {M}ath. {S}tatist. and {P}robability ({B}erkeley, {C}alif.,
  1965/66), {V}ol. {II}: {C}ontributions to {P}robability {T}heory, {P}art 1,
  Univ. California Press, Berkeley, Calif., 1967, pp.~31--42. \MR{0212152}

\bibitem{Gyongy1988}
I.~Gy\"{o}ngy, \emph{On the approximation of stochastic differential
  equations}, Stochastics \textbf{23} (1988), no.~3, 331--352. \MR{959118}

\bibitem{Gyongy1988b}
\bysame, \emph{On the approximation of stochastic partial differential
  equations. {I}}, Stochastics \textbf{25} (1988), no.~2, 59--85. \MR{999363}

\bibitem{Gyongy1989}
\bysame, \emph{The stability of stochastic partial differential equations and
  applications. {T}heorems on supports}, Stochastic partial differential
  equations and applications, {II} ({T}rento, 1988), Lecture Notes in Math.,
  vol. 1390, Springer, Berlin, 1989, pp.~91--118. \MR{1019596}

\bibitem{Gyongy1994}
\bysame, \emph{On the support of the solutions of stochastic differential
  equations}, Teor. Veroyatnost. i Primenen. \textbf{39} (1994), no.~3,
  649--653. \MR{1347193}

\bibitem{GyongyProhle1990}
I.~Gy\"{o}ngy and T.~Pr\"{o}hle, \emph{On the approximation of stochastic
  differential equation and on {S}troock-{V}aradhan's support theorem}, Comput.
  Math. Appl. \textbf{19} (1990), no.~1, 65--70. \MR{1026782}

\bibitem{Hormander1967a}
Lars H{\"o}rmander, \emph{Hypoelliptic second order differential equations},
  Acta Math. \textbf{119} (1967), 147--171. \MR{0222474 (36 \#5526)}

\bibitem{IkedaNakaoYamato1977}
Nobuyuki Ikeda, Shintaro Nakao, and Yuiti Yamato, \emph{A class of
  approximations of {B}rownian motion}, Publ. Res. Inst. Math. Sci. \textbf{13}
  (1977/78), no.~1, 285--300. \MR{0458587}

\bibitem{IkedaWatanabe1989}
Nobuyuki Ikeda and Shinzo Watanabe, \emph{Stochastic differential equations and
  diffusion processes}, second ed., North-Holland Mathematical Library,
  vol.~24, North-Holland Publishing Co., Amsterdam, 1989. \MR{MR1011252
  (90m:60069)}

\bibitem{Kunita1974}
Hiroshi Kunita, \emph{Diffusion processes and control systems}, Course at
  University of Paris VI (1974).

\bibitem{Kunita1978}
\bysame, \emph{Supports of diffusion processes and controllability problems},
  Proceedings of the {I}nternational {S}ymposium on {S}tochastic {D}ifferential
  {E}quations ({R}es. {I}nst. {M}ath. {S}ci., {K}yoto {U}niv., {K}yoto, 1976),
  Wiley, New York-Chichester-Brisbane, 1978, pp.~163--185. \MR{536011}

\bibitem{LedouxQianZhang2002}
M.~Ledoux, Z.~Qian, and T.~Zhang, \emph{Large deviations and support theorem
  for diffusion processes via rough paths}, Stochastic Process. Appl.
  \textbf{102} (2002), no.~2, 265--283. \MR{1935127}

\bibitem{MilletSanzSole1994a}
Annie Millet and Marta Sanz-Sol\'e, \emph{A simple proof of the support theorem
  for diffusion processes}, S\'eminaire de {P}robabilit\'es, {XXVIII}, Lecture
  Notes in Math., vol. 1583, Springer, Berlin, 1994, pp.~36--48. \MR{1329099}

\bibitem{NakaoYamato1978}
Shintaro Nakao and Yuiti Yamato, \emph{Approximation theorem on stochastic
  differential equations}, Proceedings of the {I}nternational {S}ymposium on
  {S}tochastic {D}ifferential {E}quations ({R}es. {I}nst. {M}ath. {S}ci.,
  {K}yoto {U}niv., {K}yoto, 1976), Wiley, New York-Chichester-Brisbane, 1978,
  pp.~283--296. \MR{536015}

\bibitem{StroockVaradhan1972}
Daniel~W. Stroock and S.~R.~S. Varadhan, \emph{On the support of diffusion
  processes with applications to the strong maximum principle}, Proceedings of
  the {S}ixth {B}erkeley {S}ymposium on {M}athematical {S}tatistics and
  {P}robability ({U}niv. {C}alifornia, {B}erkeley, {C}alif., 1970/1971), {V}ol.
  {III}: {P}robability theory, 1972, pp.~333--359. \MR{0400425}

\bibitem{WongZakai1965}
Eugene Wong and Moshe Zakai, \emph{On the relation between ordinary and
  stochastic differential equations}, Internat. J. Engrg. Sci. \textbf{3}
  (1965), 213--229. \MR{0183023}

\end{thebibliography}
\end{document}